\documentclass[11pt]{article}

\usepackage{amsfonts,amsmath,latexsym,color,epsfig}
\newtheorem{theorem}{Theorem}
\newtheorem{lemma}{Lemma}

\newenvironment{proof}
      {\medskip\noindent{\bf Proof:}\hspace{1mm}}
      {\hfill$\Box$\medskip}

\def\qed{\ifvmode\mbox{ }\else\unskip\fi\hskip 1em plus 10fill$\Box$}
\input{epsf}

\makeatletter
\def\Ddots{\mathinner{\mkern1mu\raise\p@
\vbox{\kern7\p@\hbox{.}}\mkern2mu
\raise4\p@\hbox{.}\mkern2mu\raise7\p@\hbox{.}\mkern1mu}}
\makeatother

\title{\vspace{-0.7cm}An improved bound for the stepping-up lemma}
\author{David Conlon\thanks{St John's College, Cambridge, United Kingdom.
E-mail: {\tt D.Conlon@dpmms.cam.ac.uk}. Research supported by a
research fellowship at St John's College.} \and Jacob
Fox\thanks{Department of Mathematics, Princeton, Princeton, NJ.
Email: {\tt jacobfox@math.princeton.edu}. Research supported by an
NSF Graduate Research Fellowship and a Princeton Centennial
Fellowship.} \and Benny Sudakov\thanks{Department of Mathematics,
UCLA,  Los Angeles, CA 90095. Email: {\tt bsudakov@math.ucla.edu}. Research
supported in part by NSF CAREER award DMS-0812005 and by
USA-Israeli BSF grant.}}
\date{}
\begin{document}
\maketitle

\begin{abstract}
The partition relation $N \rightarrow (n)_{\ell}^k$ means that whenever the
$k$-tuples of an $N$-element set are $\ell$-colored,
there is a  monochromatic set of size $n$, where a set is called monochromatic
if all its $k$-tuples have the same color. The logical negation of
$N \rightarrow (n)_{\ell}^k$  is written as $N \not \rightarrow  (n)_{\ell}^k$.
An ingenious construction of Erd\H{o}s and Hajnal known as the stepping-up
lemma gives a negative partition relation for higher uniformity from one of
lower uniformity, effectively gaining an exponential in each application.
Namely, if $\ell \geq 2$, $k \geq 3$, and $N \not \rightarrow  (n)_{\ell}^k$,
then $2^N \not \rightarrow  (2n+k-4)_{\ell}^{k+1}$.
In this note we give an improved construction for $k \geq 4$. We introduce a
general class of colorings which extends the framework of Erd\H{o}s and Hajnal and can be
used to establish negative partition relations. We show that if $\ell
\geq 2$, $k \geq 4$ and $N \not \rightarrow  (n)_{\ell}^k$,
then $2^N \not \rightarrow  (n+3)_{\ell}^{k+1}$. If also $k$ is odd or $\ell
\geq 3$, then we get the better bound
$2^N \not \rightarrow  (n+2)_{\ell}^{k+1}$.  This improved bound gives a
coloring of the $k$-tuples
whose largest monochromatic set is a factor $\Omega(2^{k})$ smaller than given
by the original version of the stepping-up lemma.
We give several applications of our result to lower bounds on hypergraph Ramsey
numbers. In particular, for fixed $\ell \geq 4$ we determine up to an absolute constant factor (which is independent of $k$) the size of the largest guaranteed monochromatic set in an $\ell$-coloring of the $k$-tuples of an $N$-set.
\end{abstract}

\section{Introduction}

The partition relation $N \rightarrow (n)_{\ell}^k$ means that whenever the
$k$-tuples of an $N$-element set are $\ell$-colored, there is a  monochromatic
set of size $n$, where a set is called monochromatic if all its $k$-tuples have
the same color. If $\ell=2$ we simply write
$N \rightarrow (n)^k$ instead of $N \rightarrow (n)_2^k$.

The {\it Ramsey number} $r(n)$ is the least integer $N$ such that $N
\rightarrow (n)^2$. That is,  $r(n)$ is the least integer $N$ such that every
$2$-coloring of the edges of the complete graph on $N$ vertices contains a
monochromatic clique of size $n$.
Ramsey's theorem states that $r(n)$ exists for all $n$. Determining or
estimating Ramsey numbers is one of the
central problem in combinatorics (see the book Ramsey theory
\cite{GRS90} for details). A classical result of Erd\H{o}s and
Szekeres~\cite{ES35}, which is a quantitative version of Ramsey's
theorem, implies that $r(n) \leq 2^{2n}$ for every positive
integer $n$. Erd\H{o}s~\cite{E47} showed using probabilistic
arguments that $r(n) > 2^{n/2}$ for $n
> 2$. Over the last sixty years, there have been several
improvements on these bounds (see, e.g., \cite{C08}). However,
despite efforts by various researchers, the constant factors in
the above exponents remain the same.

Although already for graph Ramsey numbers there are significant
gaps between lower and upper bounds, our knowledge of hypergraph
Ramsey numbers is even weaker. The Ramsey number $r_k(n)$ is the
minimum $N$ such that $N \rightarrow (n)^k$. That is, $r_k(n)$ is the least $N$
such that every $2$-coloring of the $k$-tuples of an $N$-element set
contains a monochromatic set of size $n$, where a set is called monochromatic
if all its $k$-tuples have the same color.
Erd\H{o}s, Hajnal, and Rado \cite{EHR65} showed that there are positive
constants $c$ and $c'$ such that $$2^{cn^2}<r_3(n)<2^{2^{c'n}}.$$
They also conjectured that $r_3(n)>2^{2^{cn}}$ for some constant
$c>0$ and Erd\H{o}s (see, e.g. \cite{CG98}) offered a \$500 reward for a proof.
Similarly,
for $k \geq 4$, there is a difference of one exponential between
known upper and lower bounds for $r_k(n)$, i.e., $$t_{k-1}(c2^{-k}n^2)
\leq r_k(n) \leq t_k(c'n),$$ where the tower function $t_k(x)$ is
defined by $t_1(x)=x$ and $t_{i+1}(x)=2^{t_i(x)}$.

The proof of the lower bound is a corollary of the following lemma of
Erd\H{o}s and Hajnal, known as the stepping-up lemma (see e.g. \cite{GRS90}).

\begin{theorem}[Stepping-up Lemma] \label{stepuporiginal}
If $k \geq 3$ and $N \not \rightarrow  (n)_{\ell}^k$, then $2^N \not
\rightarrow  (2n+k-4)_{\ell}^{k+1}$.
\end{theorem}

By starting with the negative partition relation $2^{n^2/6} \not \rightarrow
(n)^3$, which
follows by considering a random $2$-coloring of the triples of an $N$-set,
after $k-3$ iterations of the stepping-up lemma,
we have $t_{k-1}(n^2/6) \not \rightarrow (2^{k-3}n-k+3)^k$.

There is also a way to step-up from $k=2$ to $k=3$, but, in doing so, the number of colors jumps from
$2$ to $4$. However, the benefit from stepping up from $k=2$ is that already for four colors we know
the correct height for the tower function. In particular, Erd\H{o}s and Hajnal showed that there exist
constants $c$ and $c'$ such that
$$2^{2^{c n}} \leq r_3(n; 4) \leq 2^{2^{c' n}},$$
where $r_k(n; \ell)$ is the minimum $N$ such that every $\ell$-coloring of the $k$-tuples of an $N$-element set
contains a monochromatic subset of size $n$. The relevant variant of the stepping-up lemma is as follows.

\begin{theorem}[Stepping-up Lemma for $k=2$]
If $N \not \rightarrow  (n)^2$, then $2^N \not \rightarrow
(n+1)_4^{3}$.
\end{theorem}

Using the negative partition relation $2^{n/2} \not \rightarrow (n)^2$, this
gives $t_3(n/2) \not \rightarrow (n+1)_4^3$.
Through $k-3$ applications of Theorem \ref{stepuporiginal}, we have $t_{k}(n/2)
\not \rightarrow (2^{k-3}(n+1)-k+3)_4^k$.
Thus the four-color Ramsey number satisfies
\begin{equation}\label{fouroriginal}
t_k(c2^{-k}n) < r_k(n;4) < t_k(c'n),
\end{equation} where $c$ and $c'$ are absolute
constants. Satisfying as this result may be, one rather annoying aspect remains: the
exponential dependence on $k$ of the factor of $n$ in the lower bound.
We show here that this dependence on $k$ can be removed. In particular, for fixed $\ell \geq 4$, up to an absolute
constant factor, we know the size of the largest guaranteed monochromatic set an $\ell$-coloring
of the $k$-tuples of an $N$-set must have. We accomplish this by improving the
bound in the stepping-up lemma.

\begin{theorem}\label{newtheorem}
Suppose $k \geq 4$ and $N \not \rightarrow  (n)_{\ell}^k$. We have $2^N \not
\rightarrow  (n+3)_{\ell}^{k+1}$.
If $k$ is odd or $\ell \geq 3$, we have the better bound $2^N \not \rightarrow
(n+2)_{\ell}^{k+1}$.
\end{theorem}

This theorem implies that for all $k$ and $n \geq 3k$ the four-color Ramsey number satisfies
\begin{equation} \label{4colornewlowerbound}
r_k(n;4) > t_k(cn),
\end{equation}
where $c>0$ is an absolute constant. This bound improves on the lower bound in (\ref{fouroriginal}) and is tight apart from the absolute constant factor $c$. The lower bound in (\ref{fouroriginal}) for $k = 4$ shows that $N \not \rightarrow (n/3)_4^4$, where $N = t_4(cn)$
and $c>0$ is an absolute constant. After $k-4$ iterations of Theorem \ref{newtheorem}, we get $t_{k-4}(N) \not \rightarrow (n/3+ 2(k-4))_4^k$. Substituting in $n \geq 3k$ and $N=t_4(cn)$, we obtain (\ref{4colornewlowerbound}). The lower bound (\ref{4colornewlowerbound}) also holds for any number $\ell \geq 4$ of colors instead of $4$ as trivially $r_k(n;\ell) \geq r_k(n;\ell')$ if $\ell \geq \ell'$, and the upper bound in
(\ref{fouroriginal}) holds with the same
proof for any fixed number $\ell \geq 4$ of colors instead of $4$.
Theorem \ref{newtheorem} also allows us to improve the lower bound for two
colors for $n \geq 3k$ to $$r_k(n) \geq t_{k-1}(cn^2),$$ where $c>0$ is an absolute constant.

The off-diagonal partition relation $N \rightarrow (n_i)_{i<\ell}^k$ means that
whenever the $k$-tuples of an $N$-element set are $\ell$-colored with
colors $0,1,\ldots,\ell-1$, there is a color $i$ and a monochromatic set in color $i$ of size $n_i$.
The off-diagonal version of the stepping-up lemma of Erd\H{o}s and Hajnal asserts that if $k \geq 3$ and
$N \not \rightarrow (n_i)_{i<\ell}^k$, then $2^N \not \rightarrow (2n_i+k-4)_{i<\ell}^{k+1}$.
Our main result is the following extension of Theorem \ref{newtheorem}.

\begin{theorem}\label{newtheorem1}
Suppose $k \geq 4$ and $N \not \rightarrow  (n_i)_{i<\ell}^k$.
\begin{enumerate}
\item If $\ell \geq 2$, letting $n_i'=n_i+3$ for $i=0,1$ and $n_i'=n_i+1$ for
$2 \leq i < \ell$, then
$2^N \not \rightarrow  (n_i')_{i<\ell}^{k+1}$.
\item If $\ell \geq 2$ and $k$ is odd, letting $n_i'=n_i+2$ for $i=0,1$ and
$n_i'=n_i+1$ for $2 \leq i < \ell$, then
$2^N \not \rightarrow  (n_i')_{i<\ell}^{k+1}$.
\item If $\ell \geq 3$, letting $n_i'=n_i+2$ for $i=0,1,2$, and $n_i'=n_i+1$
for $3 \leq i < \ell$, then
$2^N \not \rightarrow  (n_i')_{i<\ell}^{k+1}$.
\end{enumerate}
\end{theorem}

Off-diagonal hypergraph Ramsey numbers have also received much interest.
The Ramsey-number $r_k(s,n)$ is the minimum $N$ such that every red-blue
coloring of the $k$-tuples of an $N$-set contains a red $s$-set or
a blue $n$-set. For $k=2$, after several successive
improvements, it is known (see \cite{AKS80}, \cite{BK09}, \cite{K95},
\cite{S77}) that there are constants $c_1, c_2$ such that for fixed $s \geq 3$,
\begin{equation}
\label{eq2} c_1\frac{n^{(s+1)/2}}{\log^{t} n} \leq
r_2(s,n) \leq c_2\frac{n^{s-1}}{\log^{s-2} n},
\end{equation}
where $t=(s^2-s-4)/(2s-4)$.

Erd\H{o}s and Hajnal \cite{EH72} gave a simple construction showing that
$r_3(4,n) \geq 2^{cn}$. They further conjectured that
$\lim_{n \to \infty} \frac{\log r_3(4,n)}{n}=\infty$. The authors \cite{CFS08}
recently settled this conjecture. They also improved on the lower and upper bounds for $r_3(s,n)$.
For $s=4$, they show that there are positive constants $c_1$ and $c_2$ such that
$$n^{c_1n} \leq r_3(4,n) \leq n^{c_2n^2}.$$
Using the off-diagonal stepping-up lemma of Erd\H{o}s and Hajnal, it follows from the above result that for $s \geq
2^{k-1}-k+3$, $$r_k(s,n) \geq t_{k-1}(c2^{-k}n\log n).$$ In the other direction, it follows from the results in \cite{CFS08} that $r_k(s,n) \leq t_{k-1}(c'n^{s-k+1}\log n)$ for $s \geq k+1$ and $n$ sufficiently large. This leads to the
following interesting question: what is the minimum $s=s(k)$ for which $r_k(s,n)$ grows
at least as a tower function of height $k-1$ in $n$? As in the cases $k=2,3$, it is natural to conjecture that $s(k)=k+1$. However,
the original stepping-up lemma only gives $s(k) \leq 2^{k-1}-k+3$.
 Our improved stepping-up lemma gives the linear upper bound $s(k) \leq \lceil
\frac{5}{2}k \rceil -3$ for $k \geq 4$. Indeed, starting from $s(4) \leq 7$, we have $s(k+1) \leq s(k)+3$ from the first part of Theorem \ref{newtheorem1}, and we have the better bound $s(k+1) \leq s(k)+2$ for $k$ odd from the second part of Theorem \ref{newtheorem1}. This is a step toward the conjectured bound $s(k)=k+1$.

Partition relations with infinite cardinals have an important role in modern set
theory and there is an analogous stepping-up lemma for partition relations
with infinite cardinals. The still open problem of improving the bound given by
the stepping-up lemma for infinite cardinals was raised by Erd\H{o}s and Hajnal
(see \cite{EH71}). However, we do not investigate this problem here.

\section{Step-Up Colorings}

Here we construct a large family of colorings from which we derive
negative partition relations in the next section. The coloring given by
Erd\H{o}s and Hajnal to prove the stepping-up lemma is a particularly simple coloring in this
family.

Suppose $\phi:[N]^k \rightarrow \{0,\ldots,\ell-1\}$ is an $\ell$-coloring of
the $k$-tuples of
an $N$-set such that for $0 \leq i \leq \ell-1$, there is no monochromatic
$n_i$-set in color $i$.
Let $$T=\{(\gamma_1,\ldots,\gamma_N):\gamma_i=0 \mbox{ or } 1\}.$$

Let $\mathcal{P}$ denote the family of nonempty subsets of $\{2,\ldots,k-1\}$
and $\alpha$ a function $\alpha: \mathcal{P} \times \{0,1\} \rightarrow \{0,\ldots,\ell-1\}$.
The main goal of this section is to define a coloring $C=C_{\phi,\alpha}:T
\rightarrow \{0,\ldots,\ell-1\}$. We show in the next section that, for certain choices of $\alpha$,
if $\phi$ has no large monochromatic set, then $C$ also has no large monochromatic set.

If $\epsilon = (\gamma_1, \cdots, \gamma_N)$, $\epsilon' =
(\gamma'_1, \cdots, \gamma'_N)$ and $\epsilon \neq \epsilon'$,
define
\[\delta(\epsilon, \epsilon') = \max\{i : \gamma_i \neq
\gamma'_i\},\] that is, $\delta(\epsilon, \epsilon')$ is the
largest coordinate at which they differ. Given this, we can define
an ordering on $T$, saying that
\[\epsilon < \epsilon' \mbox{ if } \gamma_i = 0, \gamma'_i = 1,\]
\[\epsilon' < \epsilon \mbox{ if } \gamma_i = 1, \gamma'_i = 0,\]
where $i=\delta(\epsilon,\epsilon')$. Equivalently, associate to any $\epsilon$
the number $b(\epsilon)
= \sum_{i=1}^N \gamma_i 2^{i-1}$. The ordering then says simply
that $\epsilon < \epsilon'$ if and only if $b(\epsilon) < b(\epsilon')$.

We will further need the following two properties of the
function $\delta$ which one can easily prove.

\vspace{0.2cm}

(a) If $\epsilon_1 < \epsilon_2 < \epsilon_3$, then
$\delta(\epsilon_1, \epsilon_2) \neq \delta(\epsilon_2,
\epsilon_3)$ and

(b) if $\epsilon_1 < \epsilon_2 < \cdots < \epsilon_p$, then
$\delta (\epsilon_1, \epsilon_p) = \max_{1 \leq i \leq p-1}
\delta(\epsilon_i, \epsilon_{i+1})$.

\vspace{0.2cm}

In particular, these properties imply that there is a unique index $i$ which
achieves the maximum of $\delta(\epsilon_i, \epsilon_{i+1})$. Indeed, suppose that there are
indices $i<i'$ such that
$$\ell =\delta(\epsilon_i, \epsilon_{i+1})=\delta(\epsilon_{i'},
\epsilon_{i'+1})=\max_{1 \leq j \leq
p-1} \delta(\epsilon_j, \epsilon_{j+1}).$$
Then, by property (b) we also have that
$\ell=\delta(\epsilon_i, \epsilon_{i'})=\delta(\epsilon_{i'},
\epsilon_{i'+1})$. This contradicts
property (a) since $\epsilon_{i}<\epsilon_{i'}<\epsilon_{i'+1}$.

We are now ready to color the complete $(k+1)$-uniform hypergraph
on the set $T$. If $\epsilon_1 < \ldots < \epsilon_{k+1}$,  for $1 \leq i \leq
k$, let
$\delta_i = \delta(\epsilon_i, \epsilon_{i+1})$. Notice that $\delta_i$ is
different from $\delta_{i+1}$ by property (a).
If $\delta_1,\ldots,\delta_k$ form a monotone sequence (increasing or
decreasing), then let
$C(\{\epsilon_1,\ldots,\epsilon_{k+1}\})=\phi(\{\delta_1,\ldots,\delta_k\})$.
That is, we color this $(k+1)$-tuple of $\epsilon$'s by the
color of the $k$-tuple of the $\delta$'s. So suppose
$\delta_1,\ldots,\delta_k$ is not monotone.

For $2 \leq i \leq k-1$, we say that $i$ is a {\it local maximum} if
$\delta_{i-1} < \delta_i > \delta_{i+1}$,
and $i$ is a {\it local minimum} if $\delta_{i-1} > \delta_i < \delta_{i+1}$.
We say that $i$ is a local {\it extremum} if it is a local minimum or a local
maximum. If $\delta_1,\ldots,\delta_k$ is not monotone, then, by property (a), this sequence has a local extremum. Let $S=\{i_1,\ldots,i_d\}_<$ denote the
local extrema labeled in increasing order. The number $i_j$ is the {\it location} of the $j$th local extremum. Note that the {\it type of local
extremum} (maximum or minimum)
of $i_j$ and $i_{j+1}$ is different for $1 \leq j \leq d-1$. That is, the type
of local extremum alternates. We now can define the desired coloring:
\begin{equation}
\label{eqtwo} C(\{\epsilon_1,\ldots,\epsilon_{k+1}\})=\left\{\begin{array}{ll}
{\alpha(S,0)}&\mbox{ if $i_1$ is a local minimum}\\
{\alpha(S,1)}&\mbox{ if $i_1$ is a local maximum.}
\end{array}\right.
\end{equation}

Note that the domain of $\alpha$ has size $2^{k-1}-2$, so there are
$\ell^{2^{k-1}-2}$ different choices for $\alpha$.
That is quite a lot of possible $\alpha$ to choose from! However, some choices
are clearly more useful then others. For example,
they may give better negative partition relations, or may be simpler to define
or analyze. The coloring used by Erd\H{o}s and Hajnal in their proof of the stepping-up lemma is given by taking $\alpha$ to be the projection map onto the second coordinate. That is, they take $\alpha(S,i)=i$ for all $S \in \mathcal{P}$ and $i \in \{0,1\}$.


We next define two more particular functions $\alpha$ which we use to establish negative
partition relations. In the first coloring, the color of a $(k+1)$-tuple
is determined by the parity of the location and the type of the first local extremum. Define the
coloring $\alpha_1:\mathcal{P} \times \{0,1\} \rightarrow \{0,1,\ldots,\ell-1\}$ as follows:
$$\alpha_1(S,i) = i_1+i \pmod 2,$$
where $i_1$ is the least element of $S$. Note that $\alpha_1$ takes only the values $0$ and $1$ in its image, but the size of its range is $\ell$. If $(\delta_1,\ldots,\delta_k)$ is not monotone, the coloring $C$ of $\{\epsilon_1,\ldots,\epsilon_{k+1}\}$
we get from $\alpha_1$ is given by

\begin{equation}
\label{eqtwo1} C(\{\epsilon_1,\ldots,\epsilon_{k+1}\})=\left\{\begin{array}{ll}
{0}&\mbox{ if $i_1$ is even and a local min. or is odd and a local max.}\\
{1}&\mbox{ if $i_1$ is even and a local max. or is odd and a local min.}
\end{array}\right.
\end{equation}

There are many more colorings similar to $\alpha_1$ which we can define that are useful in establishing negative partition
relations. To do so, we will consider proper colorings of a graph $G_k$, which we define next. Let $G_k$ be the graph with
vertex set
$\{2,\ldots,k-1\} \times \{0,1\}$, where $(2,0)$ is adjacent to $(2,1)$,
$(j,i)$ is adjacent to $(j+1,i)$ for $2 \leq j \leq k-2$, $(2,0)$ is adjacent
to $(k-1,1)$, and $(2,1)$ is adjacent to $(k-1,0)$. Graph $G_k$ is
bipartite if and only if $k$ is odd, and is always three-colorable. If $k$ is
odd, the coloring $\chi(j,i)=j+i \pmod 2$ is a proper coloring, i.e.,
adjacent vertices get different colors. Let $\chi$ be a proper coloring of the
vertices of $G_k$ with colors $0,1,\ldots,\ell-1$. Define the coloring $\alpha_{\chi}:\mathcal{P} \times
\{0,1\} \rightarrow \{0,1,\ldots,\ell-1\}$ as follows:
$$\alpha_{\chi}(S,i) = \chi(i_1,i),$$
where $i_1$ is the smallest element of $S$. Note that if $k$ is odd, $\alpha_1$ is of the form $\alpha_{\chi}$, with $\chi(j,i)=j+i
\pmod 2$.

In the next section, we show that coloring $C$ with $\alpha=\alpha_1$ demonstrates the first part
of Theorem \ref{newtheorem1}. We then use coloring $C$ with $\alpha$ of the form $\alpha_{\chi}$ to establish the second and third parts of Theorem
\ref{newtheorem1}.

\section{Proof of Theorem \ref{newtheorem1}}

Here we prove Theorem \ref{newtheorem1}. We first prove part 1 of the theorem,
which states that if
$k \geq 4$, $\ell \geq 2$, and $N \not \rightarrow  (n_i)_{i<\ell}^k$, letting
$n_i'=n_i+3$ for $i=0,1$ and $n_i'=n_i+1$ for $2 \leq i < \ell$, then
$2^N \not \rightarrow  (n_i')_{i<\ell}^{k+1}$. We then prove parts 2 and 3 of
Theorem \ref{newtheorem1}.

\vspace{0.1cm}
\noindent
{\bf Proof of Theorem \ref{newtheorem1}, Part 1:}
Suppose $\phi:[N]^k \rightarrow \{0,\ldots,\ell-1\}$ is an $\ell$-coloring of
the $k$-tuples of
an $N$-set such that for $0 \leq i \leq \ell-1$, there is no monochromatic
$n_i$-set in color $i$.
Let $$T=\{(\gamma_1,\ldots,\gamma_N):\gamma_i=0 \mbox{ or } 1\}.$$ We show that
the coloring $C:T \rightarrow \{0,\ldots,\ell-1\}$
defined as $C=C_{\phi,\alpha_1}$ from the previous section demonstrates the
negative partition relation $2^N \not \rightarrow  (n_i')_{i<\ell}^{k+1}$.

We use the following simple observation several times which holds since we color each $(k+1)$-tuple of $\epsilon$'s whose $\delta$'s are
monotone by the color of the $k$-tuple of the $\delta$'s. If we have a monochromatic clique with $n$ vertices $\{\epsilon_1,\ldots,\epsilon_n\}_{<}$ labeled in increasing order whose $\delta$'s are monotone, then in coloring $\phi$ the $\delta$'s are the vertices of a monochromatic clique of size $n-1$ in the same color.

Suppose for contradiction we have a monochromatic clique with vertices
$\{\epsilon_1,\ldots,\epsilon_{n'_{j}}\}_{<}$ of size $n'_j$ in color $j$
labeled in increasing
order. From the above observation, there can be no monotone consecutive sequence of $\delta$'s of length $n_j$. In
particular, if $j>1$, since the only $(k+1)$-tuples with color $j$ are those whose string of $\delta$'s
are monotone, there is no monochromatic clique of size $n_j+1=n_j'$ in color $j$. So we may suppose $j=0$ or $1$. By
symmetry, we may suppose $j=0$. So $n'_{0}=n_0+3$. As the sequence $\delta_1,\ldots,\delta_{n_0}$ can not be monotone, there is a first local extremum $e_1 \leq n_0-1$.

{\bf Case 1:} The first local extremum $e_1$ satisfies $e_1>2$. Already in this case, we use the assumption $k \geq 4$.
We claim that, if $e_1 <
k$, then the $(k+1)$-tuples $(\epsilon_1,\ldots,\epsilon_{k+1})$ and $(\epsilon_2,\ldots,\epsilon_{k+2})$ have different colors.
On the other hand, if
$e_1 \geq k$, then the $(k+1)$-tuples $(\epsilon_{e_1-k+2},\ldots,\epsilon_{e_1+2})$ and
$(\epsilon_{e_1-k+3},\ldots,\epsilon_{e_1+3})$ have different colors. Indeed, in both cases, the first local extremum for the pair
of $(k+1)$-tuples are of the same type (minimum or maximum) but their locations differ by one and hence have different parity,
which by (\ref{eqtwo1}) implies that these $(k+1)$-tuples have different colors.

{\bf Case 2:} The first local extremum is $e_1=2$.

{\bf Case 2(a):} $3$ is a local extremum. Recall that consecutive extrema have different type and thus types of $2$ and $3$
are distinct. This implies that the $(k+1)$-tuples $(\epsilon_1,\ldots,\epsilon_{k+1})$ and $(\epsilon_2,\ldots,\epsilon_{k+2})$
have different colors. To see this, note that for each of these $(k+1)$-tuples, the first local extremum is the second $\delta$, but
these extrema are of different type, and hence by (\ref{eqtwo1}) these $(k+1)$-tuples have different colors.

{\bf Case 2(b):} $3$ is not a local extremum. As the sequence
$\delta_2,\delta_3,\ldots,\delta_{n_0+1}$ of length $n_0$ cannot be monotone,
the sequence of
$\delta$'s has a second extremum $e_2 \leq n_0$. If $e_2 < k+1$, then the
$(k+1)$-tuples $(\epsilon_2,\ldots,\epsilon_{k+2})$ and
$(\epsilon_3,\ldots,\epsilon_{k+3})$ have different colors. If $e_2 \geq k+1$, then
the $(k+1)$-tuples
$(\epsilon_{e_2-k+2},\ldots,\epsilon_{e_2+2})$ and
$(\epsilon_{e_2-k+3},\ldots,\epsilon_{e_2+3})$ have different colors.
Indeed, in either case, the first local extremum for the pair of $(k+1)$-tuples
are of the same type but their locations differ by one and
hence have different parity, which by (\ref{eqtwo1}) implies that they have different colors. This
completes the proof of part 1 of the theorem. \qed

The last two parts of Theorem \ref{newtheorem1} follow from the discussion at
the end of the previous section together with the following lemma. The proof is
similar to the previous proof.

\begin{lemma}
Suppose $\phi:[N]^k \rightarrow \{0,\ldots,\ell-1\}$ is an $\ell$-coloring of
the $k$-tuples of
an $N$-set such that for $0 \leq j \leq \ell-1$, there is no monochromatic
$n_j$-set in color $j$. Let
$\chi:\{2,\ldots,k-1\} \times \{0,1\} \longrightarrow \{0,1,\ldots,\ell-1\}$ be
a proper coloring of the vertices of the graph $G_k$
defined at the end of the previous section. Then the coloring
$C_{\phi,\alpha_{\chi}}$ defined at the end of the previous section
has no monochromatic $(n_j+2)$-set in color $j$ if $j$ is a color used by $\chi$,
and $\chi$ has no monochromatic $(n_j+1)$-set in color $j$
if $j$ is a color not used by $\chi$.
\end{lemma}

\begin{proof}
For each color $j$ not used by $\chi$, the only $(k+1)$-tuples
$\{\epsilon_1,\ldots,\epsilon_{k+1}\}_{<}$ of color $j$ are those
whose corresponding sequence of $\delta$'s is monotonic. Hence, as we already explained in the proof of the first part of
Theorem \ref{newtheorem1}, for
each color
$j$ not used by $\chi$, there is no monochromatic clique of size
$n_j+1$ in color $j$ in coloring $C_{\phi,\alpha_{\chi}}$.

So suppose for contradiction that there is a color $j$ used by $\chi$ such that
the coloring $C_{\phi,\alpha_{\chi}}$ has a monochromatic clique
$\{\epsilon_1,\ldots,\epsilon_{n_j+2}\}_{<}$ in color $j$. As we color the
$(k+1)$-tuples of $\epsilon$'s whose $\delta$'s are
monotone by the color of the $k$-tuples of the $\delta$'s, there is no
monotone consecutive sequence of $\delta$'s of length $n_j$.
As the sequence $\delta_1,\ldots,\delta_{n_j}$ cannot be monotone, there is a
first local extremum $e_1 \leq n_j-1$.

{\bf Case 1:} The first local extremum $e_1$ satisfies $e_1>2$. If $e_1 < k$,
then the $(k+1)$-tuples $(\epsilon_1,\ldots,\epsilon_{k+1})$ and
$(\epsilon_2,\ldots,\epsilon_{k+2})$ have different colors. If $e_1 \geq k$,
then the $(k+1)$-tuples
$(\epsilon_{e_1-k+2},\ldots,\epsilon_{e_1+2})$ and
$(\epsilon_{e_1-k+3},\ldots,\epsilon_{e_1+3})$ are different
colors. Indeed, in either case, the first local extremum for the pair of $(k+1)$-tuples are of the same type but their locations
differ by one and hence, since, for all $2 \leq j \leq k - 2$ and $i = 0$ or $1$, $\chi(j, i) \neq \chi(j+1,i)$,
these $(k+1)$-tuples have different colors.

{\bf Case 2:} The first local extremum is $e_1=2$.

{\bf Case 2(a):} $3$ is a local extremum. The $(k+1)$-tuples
$(\epsilon_1,\ldots,\epsilon_{k+1})$ and
$(\epsilon_2,\ldots,\epsilon_{k+2})$ have different colors. Indeed, for each of
these $(k+1)$-tuples,
the first local extremum is the second $\delta$, but they are of different type,
and hence, since $\chi(2,0) \neq \chi(2,1)$, these $(k+1)$-tuples have different colors.

{\bf Case 2(b):} $3$ is not a local extremum. As the sequence
$\delta_2,\delta_3,\ldots,\delta_{n_j+1}$ of length $n_j$ cannot be monotone,
the sequence of $\delta$'s has a second extremum $e_2 \leq n_j$. If $e_2 = n_j$,
then the $(k+1)$-tuples $(\epsilon_1,\ldots,\epsilon_{k+1})$ and
$(\epsilon_{n_j-k+2},\ldots,\epsilon_{n_j+2})$ have different colors. To see this, note that their
only local extrema are the second $\delta$ and the $(k-1)$th $\delta$,
respectively, and are of different type. Therefore, since $\chi(2,0) \neq \chi(k-1,1)$ and $\chi(2,1) \neq
\chi(k-1,0)$, these $(k+1)$-tuples have different colors. Hence $e_2<n_j$.

If $e_2 < k+1$, then the $(k+1)$-tuples $(\epsilon_2,\ldots,\epsilon_{k+2})$
and
$(\epsilon_3,\ldots,\epsilon_{k+3})$ have different colors. If $e_2 \geq k+1$, then
the $(k+1)$-tuples
$(\epsilon_{e_2-k+2},\ldots,\epsilon_{e_2+2})$ and
$(\epsilon_{e_2-k+3},\ldots,\epsilon_{e_2+3})$ are different
colors.
Indeed, in either case, the first local extremum for the pair of $(k+1)$-tuples
are of the same type but their locations differ by one, which, as in case 1, implies that they have different colors. This completes
the proof of the lemma.
\end{proof}

\end{document}